\documentclass[11pt,a4paper]{amsart}
\usepackage{amssymb,amsmath,amsfonts}

\headsep=1.2500cm \numberwithin{equation}{section}
\newtheorem{theorem}{Theorem}[section]

\newtheorem{pro}[theorem]{Proposition}
\newtheorem{cor}[theorem]{Corollary}

\theoremstyle{definition}

\newtheorem{exa}{Example}[section]
\theoremstyle{notation}

\def\ad{\mathop{\rm ad}}
\def\id{\mathop{\rm id}}

\def\B{\mathop{\rm B}}

\def\L{\mathop{\mathcal L}}
\def\Z{\mathop{\mathcal Z}}
\def\H{\mathop{\mathcal H}}
\def\B{\mathop{\mathcal B}}
\def\R{\mathop{\mathcal R}}
\def\rad{\mathop{\rm rad}}

\def\al{\mathop{\langle}}
\def\ar{\mathop{\rangle}}

\begin{document}
\title[On Amalgamated Banach algebras]
 {On Amalgamated Banach algebras}

%%% ----------------------------------------------------------------------

%%% ----------------------------------------------------------------------

   \author[H. Pourmahmood]{H. Pourmahmood Aghababa}
   \author[N. Shirmohammadi]{N. Shirmohammadi}

   \address{Department of Mathematics, University of Tabriz, Tabriz, Iran.}
   \email{h\_p\_aghababa@tabrizu.ac.ir \\ h\_pourmahmood@yahoo.com}
   \email{shirmohammadi@tabrizu.ac.ir}

\footnote{The first author is supported by University of Tabriz. }



%%% ----------------------------------------------------------------------
\maketitle
%%% ----------------------------------------------------------------------

%%% ----------------------------------------------------------------------

%%% ----------------------------------------------------------------------

\begin{abstract}
Let $A$ and $B$ be Banach algebras, $\theta: A\to B$ be a continuous Banach algebra homomorphism and $I$ be a closed ideal in $B$. Then the direct sum of $A$ and $I$ with respect to $\theta$, denoted $A\bowtie^{\, \theta}I$, with a special product becomes a Banach algebra which is called the amalgamated Banach algebra. In this paper, among other things, we compute the topological centre of $A\bowtie^{\, \theta}I$ in terms of that of $A$ and $I$. Using this, we provide a characterization of the Arens regularity of $A\bowtie^{\, \theta}I$. Then we determine the character space of $A\bowtie^{\, \theta}I$ in terms of that of $A$ and $I$. Moreover, we study the weak amenability of $A\bowtie^{\, \theta}I$.
\end{abstract}

%%% ----------------------------------------------------------------------
%%% ----------------------------------------------------------------------

\section{Introduction}

%%% ----------------------------------------------------------------------
%%% ----------------------------------------------------------------------

Let $A$ and $B$ be Banach algebras, $\theta: A\to B$ be a continuous Banach algebra homomorphism, which without loss of generality we can assume that $\|\theta\|\leq 1$, and let $I$ be a closed ideal in $B$. We consider the Banach algebra $A\bowtie^{\, \theta} I=\{(a,i): a\in A, i\in I\}$, the $l^1$-direct sum of $A$ and $I$, with the following product formula:
$$(a,i)\cdot (a',i')=(aa', \theta(a)i'+i\theta(a')+ii').$$
$A\bowtie^{\, \theta} I$ is called the {\it amalgamation of $A$ with $B$ along $I$ with respect to $\theta$}.

The algebraic version of amalgamated Banach algebras are studied by many algebraists, see for example \cite{DFF1, DFF2, DFF3, DF, SShS}.

A special case of amalgamated Banach algebras, with $I=B$, is studied by some authors, see \cite{AGhR, BD, JN}, for example. To our knowledge there are no concrete Banach algebra with this structure. While, as Example \ref{exa1} shows, many classes of concrete Banach algebras can be represented as amalgamated Banach algebras.

The organization of paper is as follows. In Section 2 of this paper, after presenting some examples of amalgamated Banach algebras we establish some primary properties of these algebras. In Sections 3 and 4, we characterize the second dual and topological centres of $A\bowtie^{\, \theta}I$ as well as its Arens regularity. In Section 5 we characterize the character space of $A\bowtie^{\, \theta}I$. Finally, Section 6 is devoted to the investigation of the weak amenability of $A\bowtie^{\, \theta}I$.

%%% ----------------------------------------------------------------------
%%% ----------------------------------------------------------------------

\section{Some Examples and Primary Properties}

%%% ----------------------------------------------------------------------
%%% ----------------------------------------------------------------------

We commence this section with listing a number of concrete Banach algebras that have the amalgamated structure.
% ----------------------------------------------------------------------
\begin{exa}\label{exa1}
\begin{enumerate}
\item[(i)] If $\theta=0$, then $A\bowtie^{\, 0}I$ is nothing but the cartesian product of $A$ and $I$.
\item[(ii)] Let $A$ be a non-unital Banach algebra. Then the {\it unitization} of $A$, i.e. $A^{\#}={\mathbb C}\oplus A$, is the amalgamation of $\mathbb C$ with $A^{\#}$ along $A$ with respect to the homomorphism $\theta: {\mathbb C}\to A^{\#}$ defined by $\theta(\lambda)=(\lambda, 0)$.
\item[(iii)] Let $A$ be a Banach algebra and $X$ be a Banach $A$-bimodule. Then the {\it module extension Banach algebra} ${\mathcal S}=A\oplus X$ is the amalgamation of $A$ with $\mathcal S$ along $X$ with respect to the injection $\theta: A\to {\mathcal S}$ defined by $\theta(a)=(a, 0)$. Notice that the class of module extension Banach algebras include the class of triangular Banach algebras.
\item[(iv)] Let $A$ be a Banach algebra and $\phi$ be a nonzero character on $A$. Then $A\bowtie^{\, \phi} {\mathbb C}$ is the Banach algebra with the underlying Banach space $A\oplus {\mathbb C}$ and with the product
$$(a, \lambda) \cdot (a', \lambda')=(aa', \phi(a)\lambda'+\phi(a')\lambda+\lambda\lambda').$$
\item[(v)] Let $A$ and $B$ be Banach algebras and let $\phi$ be a nonzero character on $A$. Then $A\bowtie^{\, \theta} B$, the amalgamation of $A$ with $B^{\#}$ along $B$ with respect to the homomorphism $\theta: A\to B^{\#}$ defined by $\theta(a)=(\phi(a), 0)$, is the Banach algebra with the underlying Banach space $A\oplus_1 B$, the $l^1$-direct sum of $A$ and $B$, and with the following product formula:
$$(a, b) \cdot (a', b')=(aa', \phi(a)b'+\phi(a')b+bb').$$
This is a known Banach algebra denoted by $A\oplus_{\phi}B$, called the {\it $\phi$-Lau product of $A$ and $B$}, see \cite{San} for example. This class includes the class of Lau algebras introduced in \cite{Lau}.
\item[(vi)] One of the other interesting examples is the {\it semidirect product} of Banach algebras. Indeed, let $B$ be a Banach algebra, $A$ be a closed subalgebra of $B$ and $I$ be a closed ideal in $B$. If $\iota: A\to B$ is  the inclusion map, the amalgamated Banach algebra $C=A \bowtie^{\, \iota} I$ is $A \ltimes I$, the semidirect product of $A$ and $I$ \cite[Page 8]{DL} (as far as we know, the term ``semidirect product'' in the theory of (commutative) Banach algebras is introduced and studied by Thomas in \cite{Tho}). We give an important class of Banach algebras which can be recognized as a semidirect product. Let $A$ be a dual Banach algebra with predual $A_{*}$ and consider $A^{**}$, the second dual of $A$ equipped with either first or second Arens product (see Section 3 for definitions). It is shown in \cite[Theorem 2.15]{DL} that $A^{**}= A \ltimes A_*^{\perp}$, where $A_*^{\perp}=\{F\in A^{**}: F=0 \ {\rm on} \ A_* \}$. We remark that every von Neumann algebra, the measure algebra $M(G)$ of a locally compact group $G$, and the second dual of an Arens regular Banach algebra are examples of dual Banach algebras. Also the measure algebra of a locally compact group $G$ has a natural semidirect product structure. In fact we have $M(G)=l^1(G) \ltimes M_c(G)$, where $l^1(G)$ and $M_c(G)$ denote the space of discrete measures and continuous measures in $M(G)$, respectively.
\end{enumerate}
\end{exa}
%----------------------------------------------------------------------
In the following proposition, we have collected some basic properties of the Banach algebra $A\bowtie^{\, \theta}I$.
%----------------------------------------------------------------------
\begin{pro}\label{p7}
Let $A\bowtie^{\, \theta}I$ be the amalgamation of $A$ with $B$ along $I$ with respect to $\theta$.
\begin{itemize}
\item[(i)] $A\cong A\times \{0\}$ is a closed subalgebra of $A\bowtie^{\, \theta}I$, $I \cong \{0\} \times I$ is a closed ideal in $A\bowtie^{\, \theta}I$ and $\frac{A\bowtie^{\, \theta} I}{I}\cong A$.
\item[(ii)] $A\bowtie^{\, \theta}I$ is commutative if and only if $A$ and $\theta(A)+I$ are commutative.
\item[(iii)] $(a,i)$ is an identity for $A\bowtie^{\, \theta}I$ if and only if $a=1_A$, $i^2=i$, $i\in {\rm Ann}_{I}(\theta(A))$ and $\theta(a)+i=1_{\theta(A)+I}$, where ${\rm Ann}_{I}(\theta(A))=\{j\in I: j\theta(a)=\theta(a)j=0 \ {\mbox{for all}} \ a\in A\}$.
\item[(iv)] $((a_{\alpha}, i_{\alpha}))_{\alpha}$ is a (bounded) left (right, or two-sided) approximate identity for $A\bowtie^{\, \theta}I$ if and only if $(a_{\alpha})_{\alpha}$ is a (bounded) left (right, or two-sided) approximate identity for $A$, $(\theta(a_{\alpha})+i_{\alpha})_{\alpha}$ is a (bounded) left (right, or two-sided) approximate identity for $\theta(A)+I$ and $i_{\alpha}\theta(a)\to 0$ for all $a\in A$.
\item[(v)] If $A\bowtie^{\, \theta}I$ is commutative, then $A\bowtie^{\, \theta} I$ is regular if and only if both $A$ and $I$ are regular (see \cite[Definition 4.2.1]{Kan}).
\item[(vi)] $A\bowtie^{\, \theta}I$ is amenable if and only if $A$ and $I$ are amenable (see \cite[Definition 2.1.9]{Run}).
\end{itemize}
\end{pro}
\begin{proof}
All of the parts (i)-(iv) can be easily checked. The part (v) follows from Theorems 4.2.6 and 4.3.8 of \cite{Kan}, since $\frac{A\bowtie^{\, \theta} I}{I}\cong A$. Finally, (vi) follows from Corollary 2.3.2, Theorem 2.3.7 and Theorem 2.3.10  of \cite{Run}.
\end{proof}
%----------------------------------------------------------------------
\begin{cor}\label{c10}
Let $A\bowtie^{\, \id}A$ be the amalgamation of $A$ with $A$ along $A$ with respect to the identity map $\id$ on $A$.
\begin{itemize}
\item[(i)] $A\bowtie^{\, \id}A$ is commutative if and only if $A$ is commutative.
\item[(ii)] $(a,b)$ is an identity for $A\bowtie^{\, \id}A$ if and only if $a=1_A$ and $b=0$.
\item[(iii)] $((a_{\alpha}, b_{\alpha}))_{\alpha}$ is a (bounded) left (right, or two-sided) approximate identity for $A\bowtie^{\, \id}A$ if and only if $(a_{\alpha})_{\alpha}$ is a (bounded) left (right, or two-sided) approximate identity for $A$ and $b_{\alpha}\to 0$.
\item[(v)] If $A$ is commutative, then $A\bowtie^{\, \id}A$ is regular if and only if $A$ is regular.
\item[(vi)] $A\bowtie^{\, \id}A$ is amenable if and only if $A$ is amenable.
\end{itemize}
\end{cor}

%%% ----------------------------------------------------------------------
%%% ----------------------------------------------------------------------

\section{The First and Second Arens Products on $(A\bowtie^{\, \theta}I)^{**}$}

%%% ----------------------------------------------------------------------
%%% ----------------------------------------------------------------------

Let $A$ be a Banach algebra and $A^*$ and $A^{**}$ be the first and second duals of $A$, respectively. Let $a\in A$, $f\in A^*$. Then $a\cdot f$ and $f\cdot a \in A^*$ are defined by
$$\langle a\cdot f, b\rangle=\langle f, ba\rangle, \quad \langle f\cdot a, b\rangle=\langle f, ab\rangle \qquad (b\in A),$$
making $A^*$ an $A$-bimodule. Similarly, $A^{**}$ is an $A$-bimodule.

There are two natural products on $A^{**}$, called the {\it first and second Arens products}, and are denoted by $\Box$ and $\Diamond$, respectively. They were introduced by Arens \cite{Are} (for more details the reader is refereed to \cite{Dal}). We recall briefly the definitions. For $f\in A^*$ and $F\in A^{**}$ define $f\cdot F\in A^*$ and $F\cdot f\in A^*$ by
$$\langle f\cdot F, a\rangle=\langle F, a\cdot f\rangle, \quad \langle F\cdot f, a\rangle=\langle F, f\cdot a\rangle \qquad (a\in A).$$
Now, for $F, G\in A^{**}$, define $F\Box G \in A^{**}$ and $F\Diamond G\in A^{**}$ by
 $$\langle F\Box G, f\rangle=\langle F, G\cdot f\rangle, \quad \langle F\Diamond G, f\rangle=\langle G, f\cdot F\rangle \qquad (f\in A^*).$$
Then $(A^{**}, \Box)$ and $(A^{**}, \Diamond)$ are Banach algebras containing $A$ as a closed subalgebra.
%----------------------------------------------------------------------
\begin{pro}\label{p1}
$(A\bowtie^{\, \theta}I)^*$ is isometrically isomorphic to $A^* \oplus_{\infty}I^*$ as Banach spaces. The isomorphism $\Psi: A^* \oplus_{\infty}I^* \rightarrow (A\bowtie^{\, \theta}I)^*$ is given by
$$\al (a,i), \Psi(f,g) \ar= f(a)+g(i) \qquad ((a,i)\in A\bowtie^{\, \theta}I, \, (f,g)\in A^* \oplus_{\infty}I^*).$$
\end{pro}
\begin{proof}
The proof is straightforward and is omitted.
\end{proof}
%----------------------------------------------------------------------
\begin{cor}\label{c1}
$(A\bowtie^{\, \theta}I)^{**}$ is isometrically isomorphic to $A^{**} \oplus_1 I^{**}$ as Banach spaces.
\end{cor}
%----------------------------------------------------------------------
Now we explore the left and right module actions of $A\bowtie^{\, \theta}I$ on $(A\bowtie^{\, \theta}I)^*$ in order to provide a characterization of the first and second Arens product on $(A\bowtie^{\, \theta}I)^{**}$.
%----------------------------------------------------------------------
\begin{theorem}\label{t1}
Let $A\bowtie^{\, \theta}I$ be the amalgamation of $A$ with $B$ along $I$ with respect to $\theta$. Then
$$((A\bowtie^{\, \theta}I)^{**}, \Box)=(A^{**}, \Box) \bowtie^{\, \theta^{**}} (I^{**}, \Box),$$
where $\theta^{**}$ is the second adjoint of $\theta$.
\end{theorem}
\begin{proof}
Let $a,b\in A$, $i,j\in I$, $f\in A^*$ and $g\in I^*$. Then $(f,g)\cdot (a, i)\in (A\bowtie^{\, \theta}I)^*$ can be calculated as follows:
$$\begin{array}{ll}
\al (b, j), (f,g)\cdot (a, i) \ar & \!\!=
\al (a, i)\cdot (b, j), (f,g) \ar \vspace{0.1cm}\\ & =
\al (ab, i \theta(b)+ \theta(a)j+ij), (f,g) \ar \vspace{0.1cm} \\ &
= \al ab, f\ar+ \al \theta(a)j, g\ar+ \al i\theta(b), g\ar+ \al ij, g\ar \vspace{0.1cm} \\ &
= \al b, f\cdot a \ar+ \al j, g \cdot \theta(a)\ar+ \al \theta(b), g\cdot i\ar+ \al j, g\cdot i\ar \vspace{0.1cm} \\ &
= \al b, f\cdot a \ar+ \al j, g \cdot \theta(a)\ar+ \al b, \theta^*(g\cdot i)\ar+ \al j, g\cdot i\ar \vspace{0.1cm} \\ &
= \al b, f\cdot a+ \theta^*(g\cdot i) \ar+ \al j, g \cdot (\theta(a)+ i)\ar,
\end{array}$$
and so
\begin{equation}\label{eq1}
(f,g)\cdot (a, i)= (f\cdot a+ \theta^*(g\cdot i), g \cdot (\theta(a)+ i)).
\end{equation}
Further, let $F_1\in A^{**}$, $F_2\in I^{**}$. Then, in order to calculate $(F_1, F_2)\cdot (f,g) \in (A\bowtie^{\, \theta}I)^*$, one has
$$\begin{array}{ll}
\al (a, i), (F_1, F_2)\cdot (f,g) \ar & \!\!=
\al (f,g) \cdot (a, i), (F_1, F_2) \ar \vspace{0.1cm}\\ & =
\al (f\cdot a+ \theta^*(g\cdot i), g \cdot (\theta(a)+ i)), (F_1,F_2) \ar \vspace{0.1cm} \\ &
= \al f\cdot a+ \theta^*(g\cdot i), F_1\ar+ \al g \cdot (\theta(a)+ i), F_2\ar  \vspace{0.1cm} \\ &
= \al a, F_1\cdot f \ar+ \al i, \theta^{**}(F_1) \cdot g\ar+ \al \theta(a), F_2\cdot g\ar +\al i, F_2\cdot g\ar \vspace{0.1cm} \\ &
= \al a, F_1\cdot f+ \theta^*(F_2\cdot g) \ar+ \al i, F_2\cdot g+ \theta^{**}(F_1) \cdot g\ar \vspace{0.1cm} \\ &
= \al (a, i), (F_1\cdot f+ \theta^*(F_2 \cdot g), F_2\cdot g+ \theta^{**}(F_1)\cdot g)\ar.
\end{array}$$
Thus
\begin{equation}\label{eq2}
(F_1, F_2)\cdot (f,g)=(F_1\cdot f+ \theta^*(F_2 \cdot g), F_2\cdot g+ \theta^{**}(F_1)\cdot g).
\end{equation}
Now for $(F_1,F_2), (G_1,G_2)\in (A\bowtie^{\, \theta}I)^{**}\cong A^{**} \oplus_1 I^{**}$ and $(f,g)\in (A\bowtie^{\, \theta}I)^*\cong A^*\oplus_{\infty} I^*$, using \eqref{eq1} and \eqref{eq2}, we have
$$\begin{array}{ll}
\al (f,g), (F_1,F_2) \Box (G_1,G_2) \ar \!\!\!\! &=
\al (G_1,G_2) \cdot (f,g), (F_1, F_2) \ar \vspace{0.1cm}\\ & =
\al (G_1\cdot f+ \theta^*(G_2\cdot g), G_2 \cdot g + \theta^{**}(G_1)\cdot g), (F_1,F_2) \ar \vspace{0.1cm} \\ &
= \al G_1\cdot f+ \theta^*(G_2\cdot g), F_1 \ar + \al G_2 \cdot g + \theta^{**}(G_1)\cdot g, F_2 \ar \vspace{0.1cm} \\ &
= \al f, F_1\Box G_1\ar +\al G_2\cdot g, \theta^{**}(F_1) \ar + \al g, F_2\Box G_2 \ar + \al \theta^{**}(G_1)\cdot g, F_2 \ar \vspace{0.1cm} \\ &
= \al f, F_1\Box G_1\ar +\al g, \theta^{**}(F_1) \Box G_2 \ar + \al g, F_2\Box G_2 \ar + \al g, F_2 \Box \theta^{**}(G_1) \ar \vspace{0.1cm} \\ &
= \al f, F_1\Box G_1\ar +\al g, \theta^{**}(F_1) \Box G_2 + F_2\Box G_2 + F_2 \Box \theta^{**}(G_1) \ar \vspace{0.1cm} \\ &
= \al (f,g), (F_1\Box G_1, \theta^{**}(F_1\Box G_1)+ \theta^{**}(F_1) \Box G_2 + F_2\Box G_2 + F_2 \Box \theta^{**}(G_1) \ar.
\end{array}$$
Therefore,
$$(F_1,F_2) \Box (G_1,G_2)=(F_1\Box G_1, \theta^{**}(F_1) \Box G_2 + F_2\Box G_2 + F_2 \Box \theta^{**}(G_1)).$$
This completes the proof.
\end{proof}
% ----------------------------------------------------------------------
Similarly, as notation in the proof of Theorem \ref{t1}, one can show that
\begin{equation}\label{eq4}
(a,i) \cdot (f,g)= (a\cdot f+ \theta^*(i\cdot g), (\theta(a)+ i) \cdot g),
\end{equation}
$$(f,g) \cdot (F_1, F_2)=(f\cdot F_1+ \theta^*(g \cdot F_2), g\cdot F_2+ g\cdot \theta^{**}(F_1)), \vspace{0.1cm}$$
$$(F_1,F_2) \Diamond (G_1,G_2)=(F_1 \Diamond G_1, \theta^{**}(F_1) \Diamond G_2 +   F_2 \Diamond \theta^{**}(G_1)+ F_2\Diamond G_2).$$
Therefore,
$$((A\bowtie^{\, \theta}I)^{**}, \Diamond)=(A^{**}, \Diamond) \bowtie^{\, \theta^{**}} (I^{**}, \Diamond).$$

%%% ----------------------------------------------------------------------
%%% ----------------------------------------------------------------------

\section{Topological Centres}

%%% ----------------------------------------------------------------------
%%% ----------------------------------------------------------------------

Let $A$ be a Banach algebra and $X$ a Banach $A$-bimodule. Then $X^{**}$ is canonically an $(A^{**}, \Box)$-bimodule ($(A^{**}, \Diamond)$-bimodule), see \cite[Page 248]{Dal}. Let $x''\in X^{**}$ and let $L_{x^{''}}, R_{x^{''}}: (A^{**}, \Box)\rightarrow X^{**}$ be the left and right multiplication operators, respectively, i.e.
$$L_{x^{''}}(a^{''})=x^{''}\Box a^{''}=\lim_{\beta}\lim_{\alpha} x_{\beta} a_{\alpha} \quad {\rm and} \quad R_{x^{''}}(a^{''})=a^{''}\Box x^{''}= \lim_{\alpha}\lim_{\beta} a_{\alpha} x_{\beta} \qquad (a^{''}\in A^{**}),$$
where $(a_{\alpha})$ and $(x_{\beta})$ are nets in $A^{**}$ and $X^{**}$, respectively, in such a way that $a_{\alpha}\to a^{''}$ in the $w^*$-topology of $A^{**}$ and  $x_{\beta}\to x^{''}$ in the $w^*$-topology of $X^{**}$. \\
Likewise let $\L_{x^{''}}, \R_{x^{''}}: (A^{**}, \Diamond)\rightarrow X^{**}$ be the left and right multiplication operators, respectively, i.e.
$${\L}_{x^{''}}(a^{''})=x^{''}\Diamond a^{''}=\lim_{\alpha}\lim_{\beta} x_{\beta} a_{\alpha} \quad {\rm and} \quad {\R}_{x^{''}}(a^{''})= a^{''}\Diamond x^{''}=\lim_{\beta}\lim_{\alpha} a_{\alpha} x_{\beta} \qquad (a^{''}\in A^{**}).$$
The left and right topological centres, $Z_A^{\ell,t}(X^{**})$ and $Z_A^{r,t}(X^{**})$ of $X^{**}$ are
$$Z^{\ell,t}_A(X^{**})=\{x'' \in X^{**}: L_{x''} = {\L}_{x''} \}=\{x'' \in X^{**}: x'' \Box a'' = x'' \Diamond a'', \ \forall a''\in A^{**}\},$$
and
$$Z^{r,t}_A(X^{**})=\{x'' \in X^{**}: R_{x''} = {\R}_{x''} \}=\{x'' \in X^{**}: a'' \Box x'' = a'' \Diamond x'', \ \forall a''\in A^{**} \},$$
respectively. Then we say that $X$ is {\it Arens regular} (as an $A$-bimodule) or {\it $A$ acts regularly on $X$} if
$$Z_A^{\ell,t} (X^{**})=Z_A^{r,t}(X^{**})=X^{**},$$
and $X$ is {\it left strongly Arens irregular} if $Z_A^{\ell,t}(X^{**})=X$, {\it right strongly
Arens irregular} if $Z_A^{r,t}(X^{**})=X$, and {\it strongly Arens irregular} if it is
both left and right strongly Arens irregular.

If $X=A$, we will use the common notation $Z_t^\ell(A^{**})$ and $Z_t^r(A^{**})$ in place of $Z^{\ell,t}_A(A^{**})$ and $Z^{r,t}_A(A^{**})$, respectively.

Now, let $B$ be a Banach algebra, $I$ be a closed ideal in $B$ and $\theta: A\to B$ be a continuous Banach algebra homomorphism. Then we define
$$Z^\ell_{\theta^{**}} (A^{**})=\{F\in Z_t^\ell(A^{**}): \theta^{**}(F) \in Z^\ell_I(\theta(A)^{**})\},$$
note that $\theta(A)^{**}=\theta^{**}(A^{**})$ (\cite[Page 251]{Dal}), where
$$Z^\ell_I(\theta(A)^{**})=\{F\in \theta(A)^{**}: L_F={\L}_F \ {\rm on} \ I^{**}\} = \{F\in \theta(A)^{**}: F\Box G=F\Diamond G, \ \forall G\in I^{**}\}.$$
%----------------------------------------------------------------------
\begin{theorem}\label{t2}
With above notation and assumptions, one has
$$Z_t^\ell((A\bowtie^{\, \theta} I)^{**})=Z_t^\ell(A^{**}\bowtie^{\, \theta^{**}} I^{**})= Z^\ell_{\theta^{**}} (A^{**}) \bowtie^{\, \theta^{**}} (Z^\ell_t(I^{**}) \cap Z^{\ell, t}_{\theta(A)}(I^{**})).$$
\end{theorem}
\begin{proof}
Let $(F_1, G_1)\in Z_t^\ell(A^{**}\bowtie^{\, \theta^{**}} I^{**})$. Then $L_{(F_1, G_1)}=\L_{(F_1, G_1)} $ if and only if for all $(F_2,G_2)\in A^{**}\bowtie^{\, \theta^{**}} I^{**}$,
$$(F_1, G_1) \Box (F_2, G_2)=(F_1, G_1) \Diamond (F_2, G_2) ,$$
if and only if
$$(F_1 \Box F_2, \theta^{**}(F_1)\Box G_2+ G_1\Box \theta^{**}(F_2)+ G_1\Box G_2)=
(F_1 \Diamond F_2, \theta^{**}(F_1)\Diamond G_2+ G_1\Diamond \theta^{**}(F_2)+ G_1\Diamond G_2), $$
if and only if $L_{F_1}=\L_{F_1}$, $L_{\theta^{**}(F_1)}=\L_{\theta^{**}(F_1)}$ on $I^{**}$, $L_{G_1}=\L_{G_1}$ on $I^{**}$ and $\theta(A)^{**}$. Hence,
$$(F_1, G_1)\in Z^\ell_{\theta^{**}} (A^{**}) \bowtie^{\, \theta^{**}} (Z^\ell_t(I^{**}) \cap Z^{\ell, t}_{\theta(A)}(I^{**})).$$
\end{proof}
%----------------------------------------------------------------------
Using above theorem we characterize Arens regularity and strong Arens irregularity of $A\bowtie^{\, \theta}I$.
%----------------------------------------------------------------------
\begin{cor}\label{c2}
$A\bowtie^{\, \theta}I$ is Arens regular if and only if $A$ and $I$ are Arens regular, $\theta(A)$ acts regularly on $I$ and $I$ acts regularly on $\theta(A)$.
\end{cor}
%----------------------------------------------------------------------
\begin{cor}\label{c3}
$A\bowtie^{\, \theta}I$ is strongly Arens irregular if and only if $A$ and $I$ are strongly Arens irregular, $I$ acts strongly irregular on $\theta(A)$ and $\theta(A)$ acts strongly irregular on $I$.
\end{cor}
%----------------------------------------------------------------------
In the following example we determine topological centres of some amalgamated Banach algebras.
%----------------------------------------------------------------------
\begin{exa}\label{exa2}
Keep the notation of Example \ref{exa1}.
\begin{enumerate}
\item[(i)] If $\theta=0$, then
$$ Z_t^{\ell}((A\oplus I)^{**})=Z_t^{\ell}((A\bowtie^{\, 0}I)^{**})= Z_t^{\ell}(A^{**}) \bowtie^{\, 0} Z^{\ell}_t(I^{**}) = Z_t^{\ell}(A^{**}) \oplus Z^{\ell}_t(I^{**}).$$
\item[(ii)] $Z^{\ell}_t((A^{\#})^{**})={\mathbb C} \oplus Z^{\ell}_t(A^{**})= Z^{\ell}_t(A^{**})^{\#}$.
\item[(iii)] (\cite{EF}) Let ${\mathcal S}=A\oplus X$ be the module extension Banach algebra corresponding $A$ and $X$. Then, by noting that $\theta$ is the canonical embedding of $A$ into $\mathcal S$, and $I=X$ with $X^2=0$, we have
$$Z^\ell_{\theta^{**}} (A^{**})=\{F\in Z_t^\ell(A^{**}): F \in Z^\ell_X(A^{**})\}= Z_t^\ell(A^{**}) \cap Z^\ell_X(A^{**}),$$
and
$$Z^\ell_t(I^{**}) \cap Z^{\ell, t}_{\theta(A)}(I^{**})=Z^\ell_t(X^{**}) \cap Z^{\ell, t}_A(X^{**})=Z^\ell_t(X^{**})=X^{**}.$$
Therefore,
$$Z_t^{\ell}({\mathcal{S}}^{**})=(Z_t^\ell(A^{**}) \cap Z^\ell_X(A^{**})) \bowtie^{\theta^{**}} X^{**},$$
that is, $Z_t^{\ell}({\mathcal{S}}^{**})$ is  the module extension Banach algebra corresponding $Z_t^\ell(A^{**}) \cap Z^\ell_X(A^{**})$ and $X^{**}$.
\item[(iv)] $Z_t^{\ell}((A\bowtie^{\, \phi} {\mathbb C})^{**})=Z_t^{\ell}(A^{**})\bowtie^{\, \phi} {\mathbb C}$. Details are similar to details of the next general case.
\item[(v)] (\cite[Corollary 2.13]{San}) For computing $Z_t^{\ell}((A\oplus_{\phi}B)^{**})$ we note that since $\theta: A\to B^{\#}$ is defined by $\theta(a)=(\phi(a), 0)=\phi(a)$, one can easily check that $\theta^{**}: A^{**} \rightarrow (B^{\#})^{**}={\mathbb C}\oplus B^{**}$ is given by $\theta^{**}(F)=(F(\phi), 0)=F(\phi)=\phi(F)$. So
$$Z^\ell_{\theta^{**}} (A^{**})= Z^\ell_{\phi} (A^{**})=Z^\ell_{t} (A^{**}),$$
and
$$Z^{\ell, t}_{\theta(A)}(I^{**})=Z^{\ell, t}_{\mathbb C}(B^{**})=B^{**}.$$
Whence
$$Z_t^{\ell}((A\oplus_{\phi}B)^{**})=Z^\ell_{t} (A^{**})  \bowtie^{\theta^{**}} Z^\ell_t(B^{**})=Z^\ell_{t} (A^{**}) \oplus_{\phi} Z^\ell_t(B^{**}).$$
\item[(vi)] If $I=B$ and $\theta$ is surjective, then
$$Z_t^{\ell}((A\bowtie^{\, \theta} B)^{**})=(Z_t^\ell(A^{**}) \cap (\theta^{**})^{-1}(Z^\ell_t(B^{**}))) \bowtie^{\rm id} Z^\ell_t(B^{**}).$$
\item[(vii)] Assume that $B=A$. Then
$$Z_t^{\ell}((A\bowtie^{\, {\rm id}} I)^{**})=(Z_t^\ell(A^{**}) \cap Z^\ell_I(A^{**})) \bowtie^{\rm id} (Z^\ell_t(I^{**})\cap Z^{\ell, t}_{A}(I^{**})).$$
\item[(ix)] $Z_t^{\ell}((A\bowtie^{\, {\rm id}} A)^{**})=Z_t^{\ell}(A^{**})\bowtie^{\, \id} Z_t^{\ell}(A^{**})$.
\item[(x)] Let $B=I=A^{**}$ and let $\iota: A \rightarrow A^{**}$ be the cononical injection. Then
$$Z^{\ell}_{A^{**}}(\iota(A)^{**})=Z^{\ell}_{A^{**}}(A^{**})=A^{**},$$
and so $Z_{\iota^{**}}^{\ell}(A^{**})=Z_t^{\ell}(A^{**})$. Also $Z^{\ell, t}_{\iota(A)}((A^{**})^{**})=Z^{\ell, t}_{A}(A^{****})=A^{****}$, and thus
$$Z_t^{\ell}((A\bowtie^{\, {\iota}} A^{**})^{**})=Z_t^{\ell}(A^{**})\bowtie^{\, \iota^{**}} Z_t^{\ell}(A^{****}).$$
\end{enumerate}
\end{exa}
%----------------------------------------------------------------------
\begin{exa}\label{exa3}
By Example \ref{exa2} in mind we have the followings:
\begin{enumerate}
\item[(i)] The Banach algebra $A\bowtie^{\, 0} I$ is Arens regular if and only if $A$ and $I$ are Arens regular.
\item[(ii)] The unitization of $A$, $A^{\#}$, is Arens regular if and only if $A$ is Arens regular.
\item[(iii)] The module extension Banach algebra ${\mathcal S}=A\oplus X$ is Arens regular if and only if $A$ is Arens regular and $A$ acts regularly on $X$.
\item[(iv)] $A\bowtie^{\, \phi} {\mathbb C}$ is Arens regular if and only if $A$ is Arens regular.
\item[(v)] $A\oplus_{\phi}B$ is Arens regular if and only if $A$ and $B$ are Arens regular.
\item[(vi)] If $I=B$ and $\theta$ is surjective, then $A\bowtie^{\, \theta} B$ is Arens regular if and only if $A$ and $B$ are Arens regular.
\item[(vii)] If $B=A$, then $A\bowtie^{\, \theta} I$ is Arens regular if and only if $A$ is Arens regular and $A$ and $I$ act regularly on each other.
\item[(ix)] The Banach algebra $A\bowtie^{\, {\rm id}} A$ is Arens regular if and only if $A$ is Arens regular.
\item[(x)] The Banach algebra $A\bowtie^{\, {\iota}} A^{**}$  is Arens regular if and only if $A$ and $A^{**}$ are Arens regular.
\end{enumerate}
\end{exa}
%----------------------------------------------------------------------
\begin{exa}\label{exa4}
Let $G$ be an infinite locally compact group. Then by \cite[Theorem 1]{LL} we have $Z_t^{\ell}(L^1(G)^{**})=L^1(G)$, and so
$$Z_t^{\ell}((L^1(G)\bowtie^{\, {\rm id}} L^1(G))^{**})=Z_t^{\ell}(L^1(G)^{**})\bowtie^{\, \id} Z_t^{\ell}(L^1(G)^{**})=L^1(G)\bowtie^{\, \id} L^1(G).$$
Therefore, $L^1(G)\bowtie^{\, \id} L^1(G)$ is strongly Arens irregular.
\end{exa}

%%% ----------------------------------------------------------------------
%%% ----------------------------------------------------------------------

\section{Characters of $A\bowtie^{\, \theta} I$}

%%% ----------------------------------------------------------------------
%%% ----------------------------------------------------------------------

In this section, first, we provide a characterization of character space of $A\bowtie^{\, \theta} I$, and then we calculate the Jacobson radical of $A\bowtie^{\, \theta} I$ when it is commutative.
%----------------------------------------------------------------------
\begin{theorem}\label{t9}
Let $\sigma(A)\neq \emptyset$ and $\overline{\theta(A)I\cup I\theta(A)}=I$. Then $\sigma(A\bowtie^{\, \theta} I)=E\cup F$, where
$$E=\{((i\cdot \psi)\circ \theta, \psi): \psi\in \sigma(I), \,  i\in I, \, \psi(i)=1\}, \quad F=\{(\phi, 0): \phi \in \sigma(A)\}.$$
Moreover, $E$ is open and $F$ is closed in $\sigma(A\bowtie^{\, \theta} I)$.
\end{theorem}
\begin{proof}
Let $(\phi, \psi)\in \sigma(A\bowtie^{\, \theta} I)$ and $(a,i), (a',i')\in A\bowtie^{\, \theta} I$. Then
\begin{equation}\label{eq9}
\begin{array}{ll}
\phi(aa')+ \psi(\theta(a)i'+ i\theta(a)+ii') & = \al (\phi, \psi), (aa', \theta(a)i'+ i\theta(a)+ii') \ar \vspace{0.1cm} \\ &
=  \al (\phi, \psi), ((a,i)\cdot (a',i') \ar \vspace{0.1cm} \\ &
= (\phi(a)+\psi(i)) (\phi(a')+\psi(i')) \vspace{0.1cm} \\ &
= \phi(a)\phi(a')+\phi(a)\psi(i')+ \psi(i)\phi(a')+\psi(i)\psi(i').
\end{array}
\end{equation}
By taking $a=a'=0$ we see that $\psi\in \sigma(I)\cup \{0\}$. Next by taking $i=i'=0$ it follows that $\phi\in \sigma(A)\cup \{0\}$. But from \eqref{eq9}, $\phi=0$ implies $\psi(\theta(a)i')+\psi(i\theta(a'))=0$ for all $a,a'\in A$ and $i,i'\in I$, from which it follows that $\psi=0$ on $\theta(A)I\cup I\theta(A)$, and hence $\psi=0$. But this is a contradiction since $(\phi, \psi)\in \sigma(A\bowtie^{\, \theta} I)$ and so $(\phi, \psi)\neq (0,0)$. Now we have two cases:

Case I: If $\psi=0$, then $(\phi, \psi)=(\phi, 0)$.

Case II: If $\psi\neq 0$, then by \eqref{eq9} we have
$$\psi(\theta(a)i')+\psi(i\theta(a'))-\psi(i)\phi(a')-\phi(a)\psi(i')=0,$$
which implies (take $a'=0$, $i=0$),
$$\psi(\theta(a)i')=\phi(a)\psi(i') \qquad (a\in A, \, i'\in I).$$
Choose $i'\in I$ such that $\psi(i')=1$, then $\phi(a)=\psi(\theta(a)i')=(i' \cdot \psi)\circ \theta (a)$ for all $a\in A$. Therefore, $(\phi, \psi)=((i' \cdot \psi)\circ \theta, \psi)$ with $\psi(i')=1$.

Since the reverse inclusion is easy to check, so we omit its proof.

Now we show that $E$ is open in the $w^*$-topology of $\sigma(A\bowtie^{\, \theta} I)$ induced from $w^*$-topology of $A^*\times I^*$. Let $((i\cdot \psi_0)\circ \theta, \psi_0) \in \sigma(A\bowtie^{\, \theta} I)$. Then there is $i_0 \in I$ in such a way that $\psi(i_0)\neq 0$. Let $\varepsilon=|\psi(i_0)|$. Then
$$\begin{array}{ll}
U \!\!\! & =\{(\phi, \psi)\in \sigma(A\bowtie^{\, \theta} I): |(\phi, \psi)(0,i_0)- ((i\cdot \psi_0)\circ \theta, \psi_0)(0,i_0)|<\varepsilon\} \hspace{0.1cm} \\ &
=\{(\phi, \psi)\in \sigma(A\bowtie^{\, \theta} I): |\psi(i_0)- \psi_0(i_0)|<\varepsilon\},
\end{array}$$
is a neighborhood of $((i\cdot \psi_0)\circ \theta, \psi_0)$ in the
$w^*$-topology of $\sigma(A\bowtie^{\, \theta} I)$. Since $(\phi,
0)\in U$ leads to the contradiction $|\psi_0(i_0)|<\varepsilon$, it
follows that $U\subseteq E$. Therefore, $E$ is open and $F$ is
closed.
\end{proof}
%----------------------------------------------------------------------
\begin{cor}\label{c9} $($\cite[Proposition 2.4]{San}$)$
Let $\sigma(A)\neq \emptyset$ and $\phi\in \sigma(A)$. Then
$$\sigma(A\oplus_{\phi} B)=\{(\varphi, 0): \varphi \in \sigma(A)\} \cup \{(\phi, \psi): \psi\in \sigma(B)\}= (\sigma(A)\times \{0\}) \cup (\{\phi\}\times \sigma(B)).$$
\end{cor}
\begin{proof}
It is enough to note that $(i\cdot \psi)\circ \theta=\theta$ if $\psi(i)=1$.
\end{proof}
%----------------------------------------------------------------------
Let $A$ be a commutative Banach algebra. The {\it radical} of $A$, $\rad A$, is the intersection of the kernels of all characters of $A$. Also $A$ is called {\it semisimple} if $\rad A=\{0\}$.
%----------------------------------------------------------------------
\begin{theorem}\label{t11}
Let $A \bowtie^{\, \theta}I$ be commutative, $\sigma(A)\neq \emptyset$ and $\overline{\theta(A)I}=I$. Then $\rad(A\bowtie^{\, \theta} I)=\rad A\oplus \rad I$.
\end{theorem}
\begin{proof}
Let $(a,i)\in \rad(A\bowtie^{\, \theta} I)$. Then for each $\phi\in \sigma(A)$, $\phi(a)=(\phi, 0)(a,i)=0$, that is, $a\in \rad A$. Now let $\psi\in \sigma(I)$ and $\psi(j)=1$ for some $j\in I$. Then $(j\cdot \psi)\circ \theta$ belongs to $\sigma(A)$ and so $(j\cdot \psi)\circ \theta(a)=0$. Hence
$$\psi(i)=((j\cdot \psi)\circ \theta, \psi)(a,i)-(j\cdot \psi)\circ \theta(a)=0,$$
and thus $i\in \rad I$.

Conversely let $a\in \rad A$ and $i\in \rad I$. Then for each $\phi\in \sigma(A)$, $(\phi, 0)(a,i)=\phi(a)=0$ and for each $\psi\in \sigma(I)$,
$$((j\cdot \psi)\circ \theta, \psi)(a,i)=(j\cdot \psi)\circ \theta(a)+\psi(i)=0.$$
Therefore, by Theorem \ref{t9}, $(a,i)\in \rad(A\bowtie^{\, \theta} I)$.
\end{proof}
%----------------------------------------------------------------------
\begin{cor}\label{c11}
Let $A \bowtie^{\, \theta}I$ be commutative, $\sigma(A)\neq \emptyset$ and $\overline{\theta(A)I}=I$. Then $A\bowtie^{\, \theta} I$ is semisimple if and only if both $A$ and $I$ are semisimple.
\end{cor}

%%% ----------------------------------------------------------------------
%%% ----------------------------------------------------------------------

\section{Weak Amenability}

%%% ----------------------------------------------------------------------
%%% ----------------------------------------------------------------------

Let $A$ be a Banach algebra and $X$ a Banach $A$-bimodule. A {\it derivation} from $A$ into $X$ is a bounded linear map satisfying $D(ab) = a\cdot D(b) + D(a)\cdot b$ for all $a, b\in A$. For each $x\in X$ we denote by ${\ad}_x$ the derivation $D(a)=a \cdot x-x \cdot a$ for all $a\in A$, called an {\it inner derivation}. We denote by $\Z^1(A, X)$ the space of all derivations from $A$ into $X$ and by $\B^1(A, X)$ the space of all inner derivations from $A$ into $X$. The {\it first cohomology group} of $A$ with coefficients in $X$ is $\H^1(A, X)=\Z^1(A, X)/ \B^1(A, X)$. A Banach algebra $A$ is called {\it weakly amenable} if $\H^1(A,A^*)=0$.

In \cite{FM}, B. E. Forrest and L. W. Marcoux have investigated the weak amenability of triangular Banach algebras, and Y. Zhang has studied the weak amenability of module extension Banach algebras \cite{Zha}. Motivated by these earlier investigations, in this section, we study the weak amenability of amalgamated Banach algebra $A\bowtie^{\, \theta}I$.
% ----------------------------------------------------------------------
\begin{theorem}\label{t10}
If $A\bowtie^{\, \theta} I$ is commutative, then $A\bowtie^{\, \theta} I$ is weakly amenable if and only if $A$ and $I$ are weakly amenable.
\end{theorem}
\begin{proof}
This is immediate by Proposition \ref{p7}, \cite[Propositions 2.8.64 and 2.8.65(ii) and Theorem 2.8.69(i)]{Dal} and noting that $\frac{A\bowtie^{\, \theta} I}{I}\cong A$.
\end{proof}
%----------------------------------------------------------------------
\begin{exa}
Let $G$ be a locally compact abelian group and consider $M(G)=l^1(G) \ltimes M_c(G)$, the semidirect product of $l^1(G)$ and $M_c(G)$; see Example \ref{exa1}(iv). Since $l^1(G)$ is weakly amenable, by above theorem, $M(G)$ is weakly amenable if and only if $M_c(G)$ is weakly amenable.
\end{exa}
%----------------------------------------------------------------------
In general case, we have one direction of Theorem \ref{t10}.
%----------------------------------------------------------------------
\begin{pro}\label{p10}
If $A$ and $I$ are weakly amenable, then $A\bowtie^{\, \theta} I$ is also weakly amenable.
\end{pro}
\begin{proof}
It follows immediately from Proposition \ref{p7} and \cite[Proposition 2.8.65(ii)]{Dal}.
\end{proof}
%----------------------------------------------------------------------
The converse of above proposition does not hold in general. Indeed, it is shown in \cite{JW} that the augmentation ideal $I$ of $L^1(SL(2,{\Bbb R}))$ is not weakly amenable and that its unitization $I^{\#}$ is weakly amenable.
%----------------------------------------------------------------------
\begin{exa}
Let $G$ be a locally compact group. Since $M(G)=l^1(G) \ltimes M_c(G)$ and $l^1(G)$ is always weakly amenable, by Proposition \ref{p10}, $M(G)$ is weakly amenable, provided that $M_c(G)$ is weakly amenable.
\end{exa}
%----------------------------------------------------------------------
\begin{exa}
Let $G$ be a locally compact group. Then $l^1(G)\ltimes L^1(G)$ is weakly amenable.
\end{exa}
%----------------------------------------------------------------------
In order to prove a partial converse of Proposition \ref{p10} we first look at derivations from $A$ to $A^{*}$.
%----------------------------------------------------------------------
\begin{pro}\label{p11}
$\H^1(A,A^*)$ embeds in  $\H^1(A\bowtie^{\, \theta} I, (A\bowtie^{\, \theta} I)^*)$.
\end{pro}
\begin{proof}
Every $D\in \Z^1(A,A^*)$ defines a derivation ${\tilde D}: A\bowtie^{\, \theta} I \to (A\bowtie^{\, \theta} I)^*$ by ${\tilde{D}(a,i)}=(D(a), 0)$, and it can be easily checked that $\tilde D$ is inner if and only if $D$ is inner. It follows that the mapping $D\mapsto {\tilde D}$ induces an embedding from $ \H^1(A, A^*)$ into $\H^1(A\bowtie^{\, \theta} I, (A\bowtie^{\, \theta} I)^*)$.
\end{proof}
%----------------------------------------------------------------------
\begin{cor}\label{c12}
If $A\bowtie^{\, \theta} I$ is weakly amenable, then so is $A$.
\end{cor}
%----------------------------------------------------------------------
The following corollary has been obtained in \cite{GhL} with a different method. In fact, we have given a short proof for this result.
%----------------------------------------------------------------------
\begin{cor}\label{c13} $($\cite[Theorem 2.2]{GhL}$)$
Let $A$ be a dual Banach algebra. If $A^{**}$ is weakly amenable, then so is $A$.
\end{cor}
%----------------------------------------------------------------------
Weak amenability of module extension Banach algebras is extensively studied in \cite{Zha}. We are going to characterize the weak amenability of Banach algebras $A\bowtie^{\, \id} A$ and $A\oplus_{\phi} B$.

%----------------------------------------------------------------------

\subsection{Weak Amenability of $A\bowtie^{\, \id} A$}

%----------------------------------------------------------------------
Here, we focus on the special case $A\bowtie^{\, \id} A$.
%----------------------------------------------------------------------
\begin{pro}\label{p5}
Let $A^2$ be dense in $A$. Then $D\in \Z^1(A\bowtie^{\, \id} A, (A\bowtie^{\, \id} A)^*)$ if and only if
$$D(a,b)=(D_1(a)+ D_2(b), D_2(a)+D_2(b)) \qquad (a, b \in A),$$
for some $D_1,D_2\in \Z^1(A,A^*)$. Moreover, $D=\ad_{(f,g)}$ if and only if $D_1=\ad_f$ and $D_2=\ad_g$, where $f,g\in A^*$.
\end{pro}
\begin{proof}
Let $D: A\bowtie^{\, \id} A \to (A\bowtie^{\, \id} A)^*$ be a derivation. Then we may write
$$D(a,b)=(D_1(a)+ D_2(b), D_3(a)+D_4(b)) \qquad (a, b \in A),$$
where $D_k:A\to A^*$ $(1\leq k \leq 4)$ is a linear operator. If we use the derivation property of $D$ together with the equations \eqref{eq1} and \eqref{eq4}, we get
$$\hspace{-0.4cm} \begin{array}{ll}
& \big(D_1(a_1a_2)+D_2(a_1b_2+b_1a_2+b_1b_2), D_3(a_1a_2)+D_4(a_1b_2+b_1a_2+b_1b_2)\big) = \vspace{0.1cm} \\ &
 \big( a_1D_1(a_2)+a_1D_2(b_2)+ b_1D_3(a_2)+ b_1 D_4(b_2)+ D_1(a_1)a_2+D_2(b_1)a_2+D_3(a_1)b_2+ D_4(b_1)b_2, \vspace{0.1cm} \\ &
 (a_1 [D_3(a_2)+D_4(b_2)] + b_1[D_3(a_2)+D_4(b_2)]+[D_3(a_1)+D_4(b_1)]a_2+ [D_3(a_1)+D_4(b_1)]b_2 \big).
\end{array}$$
By setting $a_1=a_2=0$, we see that $D_1, D_3\in \Z^1(A,A^*)$. By setting $b_1=b_2=0$ and noting that $A^2$ is dense in $A$, we get $D_2=D_4\in \Z^1(A,A^*)$. Also, by choosing $a_2=b_1=0$, we obtain
$$D_2(a_1b_2)=a_1 \cdot D_2(b_2)+D_3(a_1) \cdot b_2,$$
which implies $D_2(a_1) \cdot b_2=D_3(a_1) \cdot b_2$ for all $a_1, b_2\in A$. Since $A^2$ is dense in $A$, it follows that $D_2=D_3$. The claim about inner derivations can be easily verified.
\end{proof}
%----------------------------------------------------------------------
\begin{theorem}\label{t7}
Let $A^2$ be dense in $A$. Then, as vector spaces, we have
$${\H}^1(A\bowtie^{\, \id} A, (A\bowtie^{\, \id} A)^*)\cong {\H}^1(A, A^*) \oplus {\H}^1(A, A^*).$$
\end{theorem}
\begin{proof}
Define $\varphi: {\Z}^1(A, A^*) \oplus {\Z}^1(A, A^*) \to {\Z}^1(A\bowtie^{\, \id} A, (A\bowtie^{\, \id} A)^*)$ by $\varphi(D_1, D_2)=D$, where
$$D(a,b)=(D_1(a)+ D_2(b), D_2(a)+D_2(b)) \qquad (a, b \in A).$$
The Proposition \ref{p5} shows that $\varphi$ is well defined and onto. Since $D$ is inner if and only if $D_1$ and $D_2$ are inner, according to the Proposition \ref{p5}, $\varphi$ induces the desired isomorphism.
\end{proof}
%----------------------------------------------------------------------
Since every weakly amenable Banach algebra is square dense, we have the following corollary.
%----------------------------------------------------------------------
\begin{cor}\label{c7}
The Banach algebra $A\bowtie^{\, \id} A$ is weakly amenable if and only if $A$ is weakly amenable.
\end{cor}
%----------------------------------------------------------------------
\begin{exa}
Let $A$ be a $C^*$-algebra or a group algebra of a locally compact group. Then $A\bowtie^{\, \id} A$ is weakly amenable.
\end{exa}

%----------------------------------------------------------------------

\subsection{Weak Amenability of $A\oplus_{\phi} B$}

%----------------------------------------------------------------------
Let $(a,b)\in A\oplus_{\phi} B$ and $(f,g)\in (A\oplus_{\phi} B)^*$. Then $(a,b)\cdot (f,g), (f,g)\cdot (a,b)\in (A\oplus_{\phi} B)^*$ are given by
\begin{equation}\label{eq7}
(a,b)\cdot (f,g)=(a\cdot f+g(b)\phi, \, \phi(a)g+ b\cdot g),
\end{equation}
\begin{equation}\label{eq8}
(f,g)\cdot (a,b)=(f\cdot a+g(b)\phi, \, \phi(a)g+ g\cdot b).
\end{equation}
%----------------------------------------------------------------------
\begin{pro}\label{p6}
Let $B^2$ be dense in $B$. Then $D\in \Z^1(A\oplus_{\phi} B, (A\oplus_{\phi} B)^*)$ if and only if
$$D(a,b)=(D_1(a)+ D_2(b), D_4(b)) \qquad (a\in A, b\in B),$$
such that
\begin{enumerate}
\item[(i)]{$D_1\in \Z^1(A,A^*)$,}
\item[(ii)]{$D_4\in \Z^1(B,B^*)$,}
\item[(iii)]{$D_2:B\to A$ is a bounded linear map satisfying
\begin{itemize}
\item[(1)]$a\cdot D_2(b)=D_2(b)\cdot a=\phi(a) D_2(b)$ for all $a\in A$ and $b\in B$,
\item[(2)]$D_2(bb')=\al b, D_4(b') \ar \phi+ \al b', D_4(b) \ar \phi$ for all $b,b'\in B$.
\end{itemize}}
\end{enumerate}
Moreover, $D=\ad_{(f,g)}$ if and only if $D_1=\ad_f$, $D_2=0$ and $D_4=\ad_g$ ($f\in A^*, g\in B^*$).
\end{pro}
\begin{proof}
Let $D: A\oplus_{\phi} B\to (A\oplus_{\phi} B)^*\cong A^* \oplus_{\infty} B^*$ be a derivation. Then $D$ is of the form
$$D(a,b)=(D_1(a)+ D_2(b), D_3(a)+D_4(b)) \qquad (a\in A, b\in B),$$
where $D_1:A\to A^*$, $D_2:B\to A^*$, $D_3:A\to B^*$ and $D_4:B\to B^*$ are linear operators. If we use the derivation property of $D$ together with the equations \eqref{eq7} and \eqref{eq8}, we get
$$\begin{array}{ll}
& \big( D_1(aa')+\phi(a)D_2(b')+\phi(a')D_2(b)+D_2(bb'), D_3(aa')+\phi(a)D_4(b')+\phi(a')D_4(b)+D_4(bb') \big) = \vspace{0.1cm} \\ &
 \big( aD_1(a')+aD_2(b')+ \al b, D_3(a')\ar \phi+ \al b, D_4(b')\ar \phi, \phi(a) D_3(a')+\phi(a)D_4(b')+bD_3(a')+ bD_4(b'))+ \vspace{0.1cm} \\ &
 (D_1(a)a'+D_2(b)a' + \al b', D_3(a)\ar \phi+ \al b', D_4(b)\ar \phi, \phi(a')D_3(a)+ \phi(a')D_4(b)]+ D_3(a)b'+D_4(b)b' \big).
\end{array}$$
By setting $b=b'=0$ we see that $D_1\in \Z^1(A,A^*)$ and $D_3\in \Z^1_{\phi}(A,A^*)$. Letting $a=a'=0$ one obtains $D_4\in \Z^1(B,B^*)$ and $D_2(bb')=\al b, D_4(b') \ar \phi+ \al b', D_4(b) \ar \phi$. \\
Now put $a=b'=0$. Then we get $b\cdot D_3(a')=0$ in $B^*$ which implies $D_3=0$ by density of $B^2$ in $B$. Hence $D_2(b)\cdot a'=\phi(a') D_2(b)$. Similarly,  choosing $a'=b=0$ gives $a\cdot D_2(b')=\phi(a) D_2(b')$.

Using \eqref{eq7} and \eqref{eq8} one can easily see that $D=\ad_{(f,g)}$ if and only if $D_1=\ad_f$, $D_2=0$ and $D_4=\ad_g$.
\end{proof}
%----------------------------------------------------------------------
Let $B$ be a Banach algebra. A derivation $D: B\to B^*$ is called cyclic if
$$\al b, D(b') \ar+ \al b', D(b) \ar=0 \quad \mbox{for all} \ b,b'\in B.$$
We denote by ${\Z}^1_c(B,B^*)$ the space of all cyclic derivations which includes ${\B}^1(B,B^*)$. The {\it first cyclic cohomology group} of $B$ is
${\H}^1_c(B,B^*)= {\Z}^1_c(B,B^*)/ {\B}^1(B,B^*)$.
%----------------------------------------------------------------------
\begin{theorem}\label{t8}
$\H^1(A,A^*)\oplus \H^1_c(B,B^*)$ embeds in $\H^1(A\oplus_{\phi} B, (A\oplus_{\phi} B)^*)$.
\end{theorem}
\begin{proof}
Define $\psi: \Z^1(A,A^*)\oplus \Z^1_c(B,B^*) \longrightarrow \Z^1(A\oplus_{\phi} B, (A\oplus_{\phi} B)^*)$ by $\psi(D_1, D_2)=D$, where
$$D(a,b)=(D_1(a), D_4(b)) \qquad (a\in A, b\in B).$$
It follows from Proposition \ref{p6} that $D$ is a derivation and it is inner if and only if $D_1$ and $D_2$ are inner. So $\psi$ induces an injective linear map from $\H^1(A,A^*)\oplus \H^1_c(B,B^*)$ into $\H^1(A\oplus_{\phi} B, (A\oplus_{\phi} B)^*)$.
\end{proof}
%----------------------------------------------------------------------
Corollary 5.6 of \cite{JW} shows that in general $\H^1(B,B^*)$ does not embeds into  $\H^1(A\oplus_{\phi} B, (A\oplus_{\phi} B)^*)$, and thus it seems that Theorem \ref{t8} be the best that one could expect.
%----------------------------------------------------------------------
\begin{cor}\label{c8} $($\cite[Theorem 2.11]{San}$)$
If $A\oplus_{\phi} B$ is weakly amenable, then $A$ is weakly amenable and $B$ is cyclicly amenable.
\end{cor}

%%% ----------------------------------------------------------------------
%%% ----------------------------------------------------------------------

%%% ----------------------------------------------------------------------
\providecommand{\bysame}{\leavevmode\hbox
to3em{\hrulefill}\thinspace}
%%% ----------------------------------------------------------------------

\end{document}